\documentclass[onefignum,onetabnum]{siamart171218}

\usepackage{tikz}\usetikzlibrary{patterns}
\usepackage{algpseudocode}
\usepackage{mathtools}
\usepackage{cite}
\usepackage{siunitx}
\usepackage{enumerate}
\usepackage{subfig}
\usepackage{pgfplots}
\pgfplotsset{width=.45\textwidth,compat=1.9}

\newcommand\papertitle{Rank-structured QR for Chebyshev rootfinding}
\newcommand\paperauthors{A.\ Casulli, L.\ Robol}

\newsiamremark{remark}{Remark}
\newsiamremark{hypothesis}{Hypothesis}
\crefname{hypothesis}{Hypothesis}{Hypotheses}
\newsiamthm{claim}{Claim}
\newsiamremark{assumption}{Assumption}

\headers{\papertitle}{\paperauthors}

\title{\papertitle\thanks{The authors are members of the research group
		INdAM--GNCS. This work has been partially supported by 
		the GNCS project ``Metodi low-rank per problemi di algebra lineare
		con struttura data-sparse''.}}

\author{A.\ Casulli\thanks{Scuola Normale Superiore, Pisa, P.za Cavalieri, 7, 56126 Pisa (PI), Italy
		(\email{angelo.casulli@sns.it}).}
	\and L.\ Robol\thanks{Department of Mathematics, 
		University of Pisa, L.go B. Pontecorvo, 5, 56127 Pisa (PI), Italy
		(\email{leonardo.robol@unipi.it})}.}

\usepackage{amsopn}

\ifpdf
\hypersetup{
	pdftitle={\papertitle},
	pdfauthor={\paperauthors}
}
\fi

\usepackage{amsmath,amssymb,amsfonts}
\usepackage{color}
\usepackage{bm}
\usepackage{easybmat}

\newcommand{\conj}[1]{\overline{#1}}

\newcommand{\fastqr}{\texttt{chebzeros}}

\newcommand{\norm}[1]{\lVert #1 \rVert}

\newcommand{\C}{\ensuremath{\mathbb{C}}}

\newcommand{\fl}[1]{\mathrm{fl}\left[ #1 \right]}

\usepackage[normalem]{ulem}

\begin{document}
	\maketitle 
	
	\begin{abstract}
We consider the computation of roots of polynomials expressed 
in the Chebyshev basis. We extend the QR iteration presented 
in [Eidelman, Y., Gemignani, L., and Gohberg, I., \emph{Numer. Algorithms
}, 47.3 (2008): pp. 253-273] introducing an aggressive early
deflation strategy, and showing that the rank-structure allows to 
parallelize the algorithm avoiding data dependencies which would 
be present in the unstructured QR. We exploit the particular
structure of the colleague linearization to achieve quadratic complexity and
linear storage requirements. 
The (unbalanced) QR iteration used for Chebyshev rootfinding 
does not guarantee backward stability on the polynomial coefficients, 
unless the vector of coefficients satisfy $\norm{p} \approx 1$, an
hypothesis which is almost never verified for polynomials approximating
smooth functions. 
Even though the presented method is mathematically equivalent to the unbalanced QR 
algorithm, we show that exploiting the rank structure allows to guarantee
a small backward error on the polynomial, up to an explicitly computable 
amplification factor $\hat\gamma_1(p)$, which depends on 
the polynomial under consideration. 
We show that this parameter is almost always of moderate size, 
making the method accurate on several numerical tests, 
in contrast with what happens in the unstructured unbalanced QR. 
	\end{abstract}

	
	
	\section{Introduction}
	
	The QR method \cite{francis1961qr}, also known as Francis' iteration, is the de facto standard
	in eigenvalue computations of small to medium size unstructured matrices. It is 
	the method implemented by the \texttt{eig} MATLAB command, and is
  used in most 
	generic eigenvalue routines available in mathematical software packages. 
	
	Recently, there has been an increasing amount of contributions dealing 
	with fast variants of the QR iteration for matrices endowed with
	special structures, most often of the form 
	$
	A = F + uv^*
	$ where $F$ is either Hermitian or unitary, and 
	$uv^*$ is a rank $1$ correction. Since the QR iteration
	can be described as a sequence of unitarily similar matrices
	$A^{(k+1)} = Q_{k}^* A^{(k)} Q_k$, where $Q_k^* Q_k = I$, the Hermitian-plus-rank-one (resp. unitary-plus-rank-one) structure
	is maintained throughout the iterations. Algorithms in this class typically achieve $\mathcal O(n^2)$ complexity
	in time and have $\mathcal O(n)$ storage requirements.
	
	The roots of a polynomial $p(x)$ expressed in the monomial basis, are 
  equal to the eigenvalues of the so called ``companion matrix'' whose
  entries are easily computable by using the coefficients of the 
  polynomial \cite[pp. 60-61]{aurentz2018core}:
  \begin{equation} \label{eq:moncomp}
    p(x) = \sum_{j = 0}^n p_n x^n = 0 \iff 
    \det \left( \begin{bmatrix}
      -\frac{p_{n-1}}{p_n} & \dots & -\frac{p_1}{p_n} & -\frac{p_0}{p_n} \\
      1 \\
      & \ddots \\
      && 1 \\
    \end{bmatrix} -xI \right) = 0.
  \end{equation}
  The companion
  matrix can be decomposed as the sum of a unitary matrix and a
  rank one correction. Hence, structured QR algorithms allows to 
  design quadratic complexity methods for computing 
  roots of polynomials. This has been one of the main motivations
  for developing methods in this class 
  \cite{aurentz2015fast,aurentz2018core,aurentz2018fast,aurentz2019fast,
		bini2004shifted,bini2005fast,bini2010fast}. Only recently some of 
    these methods were proved to be stable eigensolvers; however, 
    it is a much harder challenge to obtain stability with respect
    to the polynomial coefficients; in this case, 
    using a backward stable eigensolver
	is not sufficient \cite{edelman1995polynomial,de2016backward,dopico2018block}, i.e., 
  a small backward error on the companion matrix does not necessarily imply 
  a small backward error on the polynomial itself. It has been recently shown
	that some structured methods enjoy this property thanks to the fact that 
	the backward error shares the same unitary-plus-rank-one 
	structure of the companion matrix \cite{aurentz2018fast,noferini2019structured}. 
	
	Chebyshev polynomials are a key 
	tool for dealing numerically with smooth functions defined on a real
	interval, as demonstrated by 
	the success of Chebfun \cite{battles2004extension}
	and Chebfun2 \cite{townsend2013extension}; hence, the same problem
	of obtaining a fast and stable QR iteration 
	has been considered for colleague matrices, the Chebyshev analogue of 
	the companion matrix in \eqref{eq:moncomp}, which will be 
  introduced in Section~\ref{sec:structured-qr}.
  Contributions in this direction 
	can be found in 
	\cite{vandebril2010implicit,bini2005fast,eidelman2008efficient,boito2014implicit,gemignani2017fast}. Analogously to the monomial case, 
	computing the roots of polynomials expressed in the Chebyshev basis
	by eigenvalue methods is not generally backward stable even when
	the underlying eigenvalues solver has this property; the extensions
	of the results in \cite{edelman1995polynomial} is described in 
	\cite{noferini2017chebyshev,nakatsukasa2016stability}. More recently, 
	it has been shown that an eigenvalue solver capable of preserving the
	Hermitian-plus-rank-one structure in the backward error, 
	and ensuring a small relative
	backward error on both addends,
	is 
	stable on the polynomial coefficients \cite{noferini2019structured}. Unfortunately,
	to
	the best of our knowledge,
	this property has not been proved for 
	any of the aforementioned rank-structured methods
  for the Hermitian-plus-rank-one case. This paper aims at improving this situation. The method analyzed does not possess 
	this property for any polynomial, but we introduce an easily computable 
	parameter that characterizes the stability; we show that this stronger stability
	property is almost always attained for polynomials arising from approximating 
	smooth functions over $[-1, 1]$, and the parameter is effective at 
	detecting the cases where it does not. 
	
	\subsection{Main contributions}
	
	This work contains several contributions. We reconsider the
	algorithm described in \cite{eidelman2008efficient} for computing
	the eigenvalues of Hermitian-plus-rank-one matrices in the setting of 
	colleague matrices (which belong to this class). 
  
  We extend the method by describing how to perform aggressive early deflation
	working on the structured representation, and we show that---in contrast
	to what happens for the unstructured QR---using the Hermitian-plus-rank-one
	structure allows to easily parallelize the algorithm, avoiding data
	dependencies during the unitary transformations. 
  
  We show that in this
	context the proposed method is a backward stable eigensolver, 
  the same property of the
	unstructured unbalanced QR iteration, but only requires $\mathcal O(n)$ 
	storage and $\mathcal O(n^2)$ floating point operations (flops). However, 
  this is often not enough to deliver accurate results, because balancing
  plays an important role in this context, but would generally break  
  the Hermitian-plus-rank-one structure, making it not applicable. 
	
	For this reason, we introduce an explicitly computable parameter
	$\hat{\gamma}_j(p)$ that characterizes the backward stability of 
	the method on the polynomial coefficients. When this parameter is 
	moderate, the method produces a small relative backward error
	on the polynomial coefficients, even if the coefficient 
	vector $\lVert p\rVert$ has large norm, in contrast to the
	unstructured (unbalanced) QR iteration. 
	The boundedness of such parameter 
	cannot be guaranteed a priori, but holds in most of the practical cases that we have tested and is easy and inexpensive 
	to check at runtime. We demonstrate that this enables the use of
	the algorithm for rootfinding of analytic functions over a real interval, 
	when coupled with an interpolation step similar to the one in
	Chebfun \cite{battles2004extension}. 
	In particular, the obtained accuracy
	is on par with the one of 
	MATLAB's \texttt{eig} (which uses balancing), but with a much
	lower computational cost. 
	
	The FORTRAN code implementing the algorithm (single and double shift, 
	including a parallel version with
	aggressive early deflation of the single shift code) is publicly
	available at 
	\url{https://github.com/numpi/chebqr}, and is distributed with MATLAB interfaces. 
	
	\subsection{Notation}
	
  We denote by $\mathbb C^{m \times n}$ and $\mathbb R^{m \times n}$ the sets of $m \times n$ matrices with coefficients 
  in the complex and real fields, respectively. 
	
	We employ a Matlab-like notation for submatrices. For instance, given 
	$A \in \mathbb C^{m \times n}$ the matrix $A_{i_1:i_2, j_1:j_2}$ 
	is the submatrix obtained selecting only rows from $i_1$ to $i_2$ and 
  columns from $j_1$ to $j_2$ (extrema included).

  The symbol $\epsilon_m$ is used to indicate the unit roundoff.
  In the remainder of the paper floating point computations are 
  performed in IEEE 754 \texttt{binary64} \cite{IEEE754-2008},
  called double precision before \cite{IEEE754-1985}, for 
  which the unit-roundoff is approximately 
  $1.11 \cdot 10^{-16}$ \cite{higham2002accuracy}. The norms are denoted by $\norm{\cdot}$. When not explicitly
	specified, we mean the spectral norm (for matrices), or the Euclidean one (for vectors).
	
	Often we deal with error bounds obtained through backward error 
	analysis, for which we ignore the constants and the (polynomial) 
	dependency in the size of the problem. We denote backward errors 
	on $A$ as $\delta A$, and we write the bounds with constants
	omitted as follows:
	\[
	  \norm{\delta A} \lesssim \norm{A} \epsilon_m \iff 
	  \norm{\delta A} \leq p(m, n) \cdot \norm{A} \epsilon_m, 
	\]
	where $p(m,n)$ is a low-degree polynomial in $m,n$, the 
  dimensions of $A$; in most cases discussed in the paper, the degree of the 
  polynomial will be between $1$ and $3$. We say that an 
  algorithm for computing the eigenvalues of $A$ 
  is \emph{backward stable} if it computes the exact eigenvalues 
  of $A + \delta A$ with 
  $\norm{\delta A} \lesssim \norm{A} \epsilon_m$, and that 
  an algorithm for computing the roots of a polynomial 
  $p(x)$ is backward stable if it computes the roots of a 
  polynomial $\delta p(x)$ where\footnote{in this case we use 
  the notation $\norm{p}$ to denote the norm of the coefficients vectors for $p(x)$.} $\norm{\delta p} \lesssim 
  \norm{p} \epsilon_m$.
	
	\section{Structured QR iteration} \label{sec:structured-qr}
	Let $A = F + uv^*$ be an $n \times n$ upper Hessenberg complex matrix 
	where $F$ is Hermitian, and $uv^*$ 
	is a rank $1$ perturbation. In
	this section we describe how to efficiently compute its eigenvalues in quadratic time. 
	Although this may be applied to any
	matrix with such structure, the
	main application that we consider
	is the computation
	of the roots of a polynomial
	\[
	  p(x) = \sum_{j  =0}^n
	    p_j T_j(x), 
	\]
	expressed in the Chebyshev basis. 
	This can be achieved by computing
	the eigenvalues of 
  the colleague matrix (originally introduced 
  by \cite{good1961colleague}, and here scaled to make it
  symmetric plus rank one): 
	\begin{equation}
	\label{eq:colleague}
	  C = \frac{1}{2}\begin{bmatrix}
	    0 & 1 \\
	    1 && \ddots \\
	    & \ddots && 1 \\
	    && 1 && \sqrt{2} \\
	    &&& \sqrt{2} \\
	  \end{bmatrix} - 
	  \frac{1}{2p_n} \begin{bmatrix}
	  1 \\ 0 \\ \vdots \\ \vdots \\ 0
	  \end{bmatrix}
	  \begin{bmatrix}
	    p_{n-1} & p_{n-2} & \dots  &  p_1 & \sqrt{2} p_0 \\
	  \end{bmatrix}
	\end{equation}
	
	In Section~\ref{Gemignani} 
	we briefly recall the structured QR iteration first proposed by 
	Eidelman et al. in \cite{eidelman2008efficient}; we will then improve
	this approach by showing how to perform {aggressive early
		deflation} in a structured way in Section~\ref{structured AeD}, and how the structure allows for an easy and efficient parallelization of the algorithm in Section~\ref{parallel}.

\subsection{Hermitian-plus-rank-one QR}\label{Gemignani}

This subsection recalls the
structured implicit QR iteration of \cite{eidelman2008efficient}; 
this is mathematically equivalent to the standard one, but an 
efficient representation for the upper Hessenberg matrix is used 
to reduce the computational and storage costs. 

To establish the notation, we first summarize how the implicit 
QR iteration works in the generic case. 
Let $A$ be an $n \times n$ irreducible upper Hessenberg matrix, 
that is such that $A_{ij} = 0$ for any $i > j + 1$, and 
$A_{ij} \neq 0$ for any $i = j+1$. 
Suppose a shift $\sigma \in \mathbb C$ has been chosen. Then, 
we compute a Givens rotation $G_1$, such that 
$G_1^* (A - \sigma I) e_1$ is a multiple of $e_1$ and 
update $A$ by performing a similarity transformation 
$G_1^* AG_1$. Doing this we break the Hessenberg form, 
indeed a bulge has been created in the entry $(3,1)$. 
The subsequent steps consist in computing Givens rotations 
$G_2,\dots,G_{n-1}$ such that the matrix 
$(G_2\dots G_{n-1})^* A(G_2\dots G_{n-1})$ is again in 
Hessenberg form. We refer to this update to a QR sweep, 
or iteration. 

If the shifts $\sigma$ are chosen appropriately, we expect that in a
few sweeps 
some subdiagonal elements of $A$ will become negligible. This allows to split the eigenvalue problem into two subproblems (the matrix is now
numerically block upper triangular). Most often, the negligible
element will be in position $(n, n-1)$; hence, we immediately identify the
element in position $(n,n)$ as an eigenvalue and reduce the problem
size to $n-1$. 

A more in depth discussion of the QR algorithm, and the role
of the shifts can be found in \cite{watkins2007matrix}. 
In such book the convergence of the algorithm is interpreted in 
terms of Krylov subspaces, in particular this allows an easy 
introduction  of the double shift  (or, more generally, multishift) QR, which
is obtained replacing the initial shifted 
column $(A - \sigma I) e_1$ with the evaluation
of a low degree polynomial $\rho(A) e_1$. 

\subsubsection{Structure preservation}

We note that each upper Hessenberg matrix $A = F + u v^*$ such that 
$F = F^*$ can be completely determined using $\mathcal O(n)$ parameter. Indeed, 
we have the following relations
for the 
entries of $F$ in the lower triangular part, excluding the first subdiagonal:
\begin{equation} \label{eq:Aij1}
  0 = A_{ij} = F_{ij} + u_i \conj{v_j} \implies F_{ij} = -u_i \conj{v_j}, \qquad 
  \forall i > j + 1, 
\end{equation}
and therefore, relying on $F_{ij} = \conj{F}_{ji}$, we may write:
\begin{equation} \label{eq:Aij2}
  A_{ij} = F_{ij} + u_i \conj{v_j} = 
  \conj{F_{ji}} + u_i \conj{v_j} =
  u_i \conj{v_j} -\conj{u_j} v_i, \qquad 
  \forall j > i + 1. 
\end{equation}
Hence, all the entries above the first superdiagonal 
can be determined solely 
from $u$ and $v$; the complete matrix $A$ can be recovered storing its
diagonal and subdiagonal entries, and the vectors $u,v$. 

In addition, performing a unitary similarity preserves the 
Hermitian-plus-rank-one structure since
\[
  QAQ^*=QFQ^*+(Qu)(Qv)^*, \qquad Q^*Q = I. 
\] 
Hence, the structure is inherited by all the
matrices produced by the QR iteration. 

Algorithmically, we store two vectors 
$d\in \C^n$ and $\beta\in \C^{n-1}$ whose entries are the diagonal and the subdiagonal entries of $A$, respectively. This completely determines the 
lower triangular part of $A$ (thanks to the upper Hessenberg form); 
the upper part can
be retrieved exploiting the symmetry as discussed above, which yields:
\begin{equation}\label{g2.2}
A_{i,j} = \begin{cases}
  d_i  & i = j \\
  \beta_{i-1} & i = j + 1 \\
  \conj{\beta_i} - \conj{u_{i+1}} v_{i} + u_{i} \conj{v_{i+1}} & j = i + 1\\
   u_{i} \conj{v_j} - \conj{u_j} v_{i} & j > i + 1\\  
\end{cases}
\end{equation} 
The preservation of Hermitian-plus-rank-one structure implies that 
a matrix obtained after $k$ sweeps of QR can be determined using $4$ vectors $d^{(k)},\beta^{(k)},u^{(k)},v^{(k)}$; we call these vectors the generators of $A^{(k)}$. 

The structured QR iteration is then implemented as a sequence
of rotation acting on consecutive rows. Each of these rotations 
can be applied to the structured representation by updating 
in place a constant number of entries in $d^{(i)}, \beta^{(i)}, u^{(i)}, v^{(i)}$.

Storing only these vectors makes the memory storage linear, in contrast to the quadratic cost for the standard QR algorithm. 

The details of the iteration are reported in sections \ref{sec:single-shift}
(for the single shift case) and \ref{sec:double-shift} (for the double shift algorithm).

\subsubsection{Single-shift iteration}
\label{sec:single-shift}
In this section we briefly summarize the
single shift iteration of the structured QR algorithm introduced in \cite{eidelman2008efficient}.
Let $A = F + uv^*$ be a
Hermitian-plus-rank-one matrix in upper Hessenberg form; 
we discussed how to store it 
using $\mathcal O(n)$ memory, and we aim at computing its eigenvalues leveraging
this representation, and achieving $\mathcal O(n^2)$ complexity. 

We consider a single shift QR step in this structured format. 
Let $\rho(z)=z-\sigma$ be a shift polynomial. As in the implicit version of the QR algorithm, to obtain the starting Givens rotator $G_1$ we  compute the first column of the matrix $A-\sigma I$. Since $A$ is in 
upper Hessenberg form we have
\begin{equation}
\rho(A)e_1=\begin{bmatrix}
d_1-\sigma\\
\beta_1\\
0\\
\vdots\\
0
\end{bmatrix}.
\end{equation}
The rotator such that $G_1^*\rho(A)e_1$ is a multiple of $e_1$ can be computed by standard BLAS routines, such as \texttt{zrotg}.

$A$ is now replaced by $G_1^* A G_1$; since $A$ is represented
by means of the four vectors $d, \beta, u, v$, we directly 
update this representation. Note that the two matrices only differ in
the top $3$ rows. Thanks to the symmetry, the updated
parameters can be recovered from the top $3 \times 2$ block of the
updated matrix, which we compute as follows:
\begin{equation} \label{eq:single-shift-dbeta}
  \left(G_1^* A G_1\right)_{1:3,1:2} = 
  \begin{bmatrix}
    \hat{G}_1^* \\
    & 1
  \end{bmatrix} \begin{bmatrix}
  d_1&{(\overline{\beta_1}-\overline{u_2}{v_1})}+u_1\bar{v}_2\\
  \beta_1&d_2\\
  0&\beta_2
  \end{bmatrix} 
  \hat{G}_1,
\end{equation}
where $\hat{G}_1=(G_1)_{1:2,1:2}.$

The $3 \times 2$ updated block can be computed directly, using a constant
number of floating point operations independent of $n$, and the the updated
entries of $d, \beta$ can be read off from the
updated matrix. Similarly, the updated $u,v$ 
can be computed by $G_1^* u$ and $G_1^* v$. 

As in the standard implicit QR, this operation creates a bulge, breaking 
the Hessenberg structure in position $(3,1)$. In particular, the matrix cannot 
be recovered from the generators unless this bulge is stored as well. Since 
this is just one extra entry, it does not affect the computational costs
or the storage. 

The bulge can be chased 
with another rotation $G_2$ operating on the rows $(2,3)$, whose action
can be computed updating the vectors $d,\beta,u,v$ in a similar fashion.
The procedure is repeated until the bulge disappears at the bottom of the 
matrix.

When applying $G_i^*$ on the left and $G_i$ on the right with $i > 1$ the analogous
 $3 \times 2$ block $A_{i:i+2,i:i+1}$ needs to be considered. We note that to compute $G_i$ the data required are just the subdiagonal element $\beta_{i-1}$ and the bulge given from the previous step. 
In Algorithm \ref{single_bulge_chasing} we summarize how to perform the $i$-th bulge chasing step.

	\begin{algorithm}
		\caption{Single-shift bulge chasing ($i$-th step)}\label{single_bulge_chasing}
		\begin{algorithmic}[1]
			\Procedure{chasing}{$d,\beta,u,v$,bulge}
			\State $\gamma\gets (\overline{\beta}_i-\overline{u}_{i+1}{v_i})+u_i\overline{v}_{i+1}$
			\State $G^*\gets\text{Givens}\left(\begin{bmatrix}
				\beta_{i-1}\\
				\text{bulge}
				\end{bmatrix}\right)$
			\State $\beta_{i-1}\gets\left\lVert\begin{bmatrix}
				\beta_{i-1}\\
				\text{bulge}
			\end{bmatrix}\right\lVert_2$
			\State$\begin{bmatrix}
				d_{i}&\gamma\\
				\beta_{i}&d_{i+1}\\
				\text{bulge}&\beta_{i+1}
			\end{bmatrix}\gets \begin{bmatrix}
			G^*&\\
			&1
			\end{bmatrix}
			\begin{bmatrix}
			d_{i}&\gamma\\
			\beta_{i}&d_{i+1}\\
			0&\beta_{i+1}
			\end{bmatrix}G$
			\State $\begin{bmatrix}
				u_{i}\\
				u_{i+1}
			\end{bmatrix}\gets G^* \begin{bmatrix}
			u_{i}\\
			u_{i+1}
		\end{bmatrix}$
		\State $\begin{bmatrix}
			v_{i}\\
			v_{i+1}
		\end{bmatrix}\gets G^* \begin{bmatrix}
			v_{i}\\
			v_{i+1}
		\end{bmatrix}$
				\EndProcedure
			\end{algorithmic}
		\end{algorithm}

The cost of each of the
above steps is constant and every iteration requires $\mathcal O(n)$
steps; hence, the cost of each iteration is linear in $n$. Assuming the convergence of all the eigenvalues in a number of steps linear in the dimension of the problem\footnote{The QR iteration typically converges in about 
	$2n$ or $3n$ iterations; this measure is only weakly influenced by
	the matrix under consideration, and is one of the features that makes
	the QR iteration the method of choice for the dense unsymmetric eigenvalue 
	problems. We refer to Section~\ref{sec:numexp} for some numerical experiments
  supporting this claim.} the total cost is quadratic in $n$.

The described procedure is backward stable in the original data: the computed eigenvalues are the exact ones of $A + \delta A$, where 
\[
  \norm{\delta A} \lesssim \left( \norm{F}_2 + \norm{u}_2 \norm{v}_2 \right) \epsilon_m, 
\]
where $\lesssim$ denotes an upper bound up to a small degree polynomial in
the dimension, as is customary for backward error analyses. This result
is proven in  \cite{eidelman2008efficient} but is in general weaker than 
the stability guaranteed by the usual unstructured QR iteration, which instead
guarantees that $\norm{\delta A} \lesssim \norm{A} \epsilon_m$. 
Without further assumptions, $\norm{F}_2 + \norm{u}_2 \norm{v}_2$ might 
be arbitrarily larger than $\norm{A}$. 

In Section~\ref{backward} we reconsider the backward error analysis
focusing on the problem of interest in this paper---the computation
of roots for Chebyshev polynomials---and we show that for this case the same
rigorous bound of the standard QR iteration can be obtained. In addition, 
we will also show that the structured method often exhibits superior stability
properties.


\subsubsection{Double shift iteration}
\label{sec:double-shift}
Whenever the matrix $A$ is real, it can be convenient to rely on a double 
shift iteration, instead of the single shift approach described in the 
previous section. This makes the method more efficient (since complex arithmetic
can be avoided completely), and ensures that the computed eigenvalue have 
the required complex-conjugation symmetry. 

From a high-level perspective, this introduces the following changes. Instead 
of choosing a degree one polynomial $\rho(z) = z - \sigma$ we allow to pick a
degree $2$ one, of the form $\rho(z)=z^2-2\Re(\sigma) z+|\sigma|^2$, where
$\sigma \in \mathbb C$. Then, we determine two rotations $G_{1,1}$ and $G_{1,2}$ 
such that $G_{1,1}^* G_{1,2}^* \rho(A) e_1$ is a multiple of $e_1$. Performing 
a similarity using $G_{1,1}^* G_{1,2}^*$ in place of $G_1$ creates a perturbation
in the upper Hessenberg structure, which is then chased to the bottom of the
matrix. This procedure is standard, and we refer the reader to
\cite{watkins2007matrix} for further details. 

In a similar fashion to what is discussed in Section~\ref{sec:single-shift}, 
this procedure can be carried relying on the structured
representation of 
the matrix $A$. 
The column
vector $\rho(A) e_1$ can be computed with a constant number of operation
in terms of $d, \beta, u, v$:
\begin{equation}
	\rho(A)e_1=\begin{bmatrix}
	d_1^2+\gamma_1\beta_1-2\Re(\alpha)d_1+ |\alpha|^2\\
	\beta_1(d_1+d_2-2\Re(\alpha))\\
	\beta_1\beta_2\\
	0\\
	\vdots\\
	0
	\end{bmatrix}.
\end{equation}
We now compute the two Givens rotations, $G_{1,1}$ and $G_{1,2}$, 
such that $G_{1,1}^* G_{1,2}^*\rho(A)e_1$ is a multiple of $e_1$. 
The matrix $G_{1,1}^* G_{1,2}^* A G_{1,2} G_{1,1}$ differs from $A$ only in the
top $3$ rows; analogously to the single shift case, we can recover the updated
parameters by only computing the top $4 \times 3$ block, and updating
$u, v$ separately. This does also provide the three entries composing 
the ``bulge'', i.e., 
the ones that perturb the upper Hessenberg structure. 

Let us denote for the sake of brevity by $\gamma_i:={(\conj{\beta_i}-\conj{u_{i+1}}{v}_i)}+u_i\bar{v}_{i+1}$ the
entries on the upper diagonal of $A$, recovered according to the procedure
described in \eqref{g2.2}.
Then, the computation required to apply the two rotation is the following:
\begin{equation} \label{eq:double-first-steps}
\begin{split}
\left(G_{1,1}^* G_{1,2}^* A G_{1,2} G_{1,1}\right)_{1:4,1:3} &= \\
=\begin{bmatrix}
\hat{G}_{1,1}^* &\\
& 1 
\end{bmatrix} 
&
\begin{bmatrix}
\hat{G}_{1,2}^* &\\&1
\end{bmatrix}\begin{bmatrix}
d_1&\gamma_1& u_1\conj{v}_3 -\conj{u_1}{v}_3 \\
\beta_1&d_2&\gamma_2\\
0&\beta_2&d_3\\
0&0&\beta_3
\end{bmatrix}\hat{G}_{1,2} \hat{G}_{1,1}, 
\end{split}
\end{equation}
where $\hat{G}_{1,1}=(G_{1,1})_{1:3,1:3}$ and  $\hat{G}_{1,2}=(G_{1,2})_{1:3,1:3}$. 
 
The cost in floating point operations of this update is independent of $n$,
and therefore it accounts for $\mathcal O(1)$ in the computational cost. 

In the updated 
matrix of \eqref{eq:double-first-steps}, the entries in position $(3,1), (4,1), (4,2)$ constitute
the bulge; they need to be chased down to the bottom of the matrix to
complete a single sweep. 

We describe in more detail the generic 
$i$-th step of bulge chasing; assume that the three entries composing the bulge
are known from the previous step, or from the initialization of the sweep
described above. We denote them by $b_{i+2,i}, b_{i+3,i}, b_{i+3,i+1}$. 
Then, we compute Givens rotations 
$G_{i+1,1}$ and $G_{i+1,2}$ such that 
\begin{equation} \label{eq:double-rotg}
\hat{G}_{i+1,1}^*
\hat{G}_{i+1,2}^*
\begin{bmatrix}
\beta_i\\
b_{i+2, i}\\
b_{i+3, i}
\end{bmatrix}=\begin{bmatrix}
c\\0\\0
\end{bmatrix},
\end{equation}
for a suitable constant $c$, where $\hat{G}_{i+1,j}=(G_{i+1,j})_{i+1:i+3,i+1:i+3}$ for $j=1,2$. The rotations can be used to build the 
matrix $A^{(i+1)}$ defined as follows:
\[
  A^{(i+1)} := G_{i+1,1}^*G_{i+1,2}^* A^{(i)} G_{i+1,2} G_{i+1,1}.
\]
To continue with the process we need to update the generators 
used to store the matrix $A^{(i)}$, and also memorize the updated
bulge, which will have moved down one step. Updating $u^{(i)}$ and $v^{(i)}$ 
requires only to multiply them on the left by $G_{i+1,1}^* G_{i+1,2}^*$, 
that is:
\begin{equation*}
u^{(i+1)}=G_{i+1,1}^* G_{i+1,2}^* u^{(i)} \quad \text{and} \quad v^{(i+1)}=G_{i+1,1}^* G_{i+1,2}^* v^{(i)}.
\end{equation*}
Computing the action of the rotation in 
\eqref{eq:double-rotg} yields the updated
value of $\beta_i$. 
To update the other entries in 
the vectors $d$ and $\beta$ 
we compute the action of the rotations on the
$4 \times 3$
submatrix $A^{(i)}_{i+1:i+4, i+1:i+3}$, which can be written explicitly
as follows:
\begin{equation} \label{eq:double-shift-chase}
\begin{bmatrix}
 \hat G_{i+1}^* \\ & 1 
\end{bmatrix}
\begin{bmatrix}
d^{(i)}_{i+1}&\gamma_{i+1}&u^{(i)}_{i+1}\bar{v}^{(i)}_{i+3} {-\conj{u}^{(i)}_{i+3}{v}^{(i)}_{i+1}}\\
\beta^{(i)}_{i+1}&d^{(i)}_{i+2}&\gamma_{i+2}\\
b_{i+3,i+1}&\beta^{(i)}_{i+2}&d^{(i)}_{i+3}\\
0&0&\beta_{i+3}
\end{bmatrix}
\hat G_{i+1},
\end{equation}
where $\hat G_{i+1} := 
 \hat G_{i+1,2} \hat G_{i+1,1}.$

After the update has been performed, the
new bulge will be available in the entries in position 
$(3,1), (4,1)$ and $(4,2)$ of the matrix in 
\eqref{eq:double-shift-chase}. These entries will form the 
values $b_{i+3,i+1}, b_{i+4,i+1}$ and $b_{i+4,i+2}$ used 
in the next chasing step. 

The cost of each step is independent of $n$, and every sweep requires $n$ steps, 
so the cost of each iteration is linear. 
Using the same argument employed in the single shift case, 
and assuming convergence within
$\mathcal O(n)$ sweeps, we conclude that the 
algorithm has a quadratic cost.

\subsection{Aggressive early deflation} 
\label{structured AeD}

Modern implementations of the QR iteration, such as the one found in 
LAPACK \cite{anderson1999lapack}, rely on a number of strategies to improve their
efficiency. One of the most effective, developed by Braman, Byers and Mathias in \cite{braman2002multishift2} is known as aggressive early deflation (AED). 
AED is a method for deflating eigenvalues that, combined with standard deflation, improves the convergence speed, and in particular speeds up the
detection of almost deflated eigenvalues. 

A practical description of this method can be found in \cite{watkins2007matrix}.
In this section we consider the case of a matrix $A = F + uv^*$, with
the now usual assumption $F = F^*$. It turns out that, if we are running 
the QR iteration (either single or double shift) in the structured 
format described in Sections~\ref{sec:single-shift} and \ref{sec:double-shift}, 
the AED step can be performed in structured form as well. 

We now briefly recall how AED works in general, and then show how these 
operations can be carried out in a structured way. Let $k\ll n$, and consider
the following partitioning of $A$:
\[
  A = \begin{bmatrix}
  A_{11} & A_{12} \\
  A_{21} & A_{22} \\
  \end{bmatrix}, 
\]
where $A_{22}$ is $k \times k$, $A_{11}$ is $(n-k) \times (n-k)$. We note that $A_{21}$ 
is made of a single entry in its top-right corner, which is equal to $A_{n-k+1,n-k}$. 

We  rely on a simple implementation of the QR iteration (i.e., without the AED scheme) to compute
the Schur form of $A_{22}$. This does not pose any computational challenges, since we assume $k$ 
to be small. Then, we use the computed Schur form $Q^* A_{22} Q = T$ to perform the following
unitary similarity transformation on the original matrix $A$:
 \begin{equation}\label{AeD1}
 \begin{bmatrix}
 A_{11}&A_{12}Q\\
 Q^*A_{21}& Q^*A_{22}Q
 \end{bmatrix},
 \end{equation}
Note that $Q^* A_{21}$ has non-zero entries only in its last column. Let us call this column $x$. Then,
the AED scheme proceeds by partitioning the vector $x$ as follows
\[
  Q^* A_{21} e_{n-k} = x = \begin{bmatrix}
    x_1 \\
    x_2 \\
  \end{bmatrix}, 
\]
where $x_1$ has $0 \leq j \leq k$ entries, and $j$ is chosen as small as possible under the 
constraint that all 
the entries of $x_2$ satisfy 
\begin{equation} \label{eq:aed-defl-crit-dense}
  |x_2|_i \leq \min\{ |T_{n-k+j+i}|, |A_{n-k+1,n-k}| \} \cdot \epsilon_m.
\end{equation}
If \eqref{eq:aed-defl-crit-dense} is satisfied for some entries, than these can be safely
set to zero without damaging the backward stability of the approach, and allowing
to deflate
$k - j$ eigenvalues at once. More precisely, it implies that $k-j$ eigenvalues of $A_{22}$ 
are also (numerically) eigenvalues of $A$, even though they could not be detected immediately
by the usual deflation criterion. The remaining eigenvalue of $A_{22}$ can be put to good 
use by employing them as shifts for the next sweeps. 

The described procedure can be applied directly in the structured format that we have
relied upon for developing the QR iteration. As a first step, we need to characterize the structure of the partitioned matrix \eqref{AeD1}. 
Submatrices of $A$ are also completely determined
by the $4$ vectors, as we summarize in the next Lemma. 

\begin{lemma} \label{lem:diagblocks}
	Let $A = F + uv^*$ be an upper Hessenberg
	Hermitian-plus-rank-one matrix, with vector generators
	$d,\beta,u,v$. Then, if we partition 
	\[
	  A = \begin{bmatrix}
	    A_{11} & A_{12} \\
	    A_{21} & A_{22} \\
	  \end{bmatrix}, \qquad 
	  A_{11} \in \mathbb C^{(n-k) \times (n-k)}, \quad 
	  A_{22} \in \mathbb C^{k \times k}, 
	\]
	then also
		$A_{11}$ and $A_{22}$ are upper Hessenberg
		Hermitian-plus-rank-one matrices, with generators
		$d_{1:n-k}$, $\beta_{1:n-k-1}$, $u_{1:n-k}$, $v_{1:n-k}$, and 
		$d_{n-k+1:n}$, $\beta_{n-k+1:n-1}$, $u_{n-k+1:n}$, $v_{n-k+1:n}$, respectively.
\end{lemma}

\begin{proof}
	Direct verification. 
\end{proof}

Lemma~\ref{lem:diagblocks} implies that the generator representation used for $A$ immediately
gives generator representations for $A_{11}$ and $A_{22}$. For the latter, we can rely on this
representation and the algorithm presented in the manuscript to construct a Schur form, and at 
the same time update using the computed rotations the last column of $A_{21}$, which is equal
to $\beta_{n-k} e_1$. Hence, we obtain a structured representation of the updated matrix: 
\[
	\begin{bmatrix}
		A_{11}&A_{12}Q\\
		Q^*A_{21}& Q^*A_{22}Q
	\end{bmatrix},
\]
Then, we proceed as in the dense case, and look at the entries of the vector $x$ composing the
last column of $Q^{*} A_{21}$. If it contains negligible elements, we set them to zero and deflate
the associated components. Note that the check in \eqref{eq:aed-defl-crit-dense} can be 
stated directly in terms of the generators as follows:
\begin{equation}
  |x_2|_i <\min \{|\beta_{n-k}|,|d_{n-k+i+j}|\} \epsilon_m.
\end{equation}
Once the deflated components have been removed, we are left with a matrix for which the top
$(n-k-1) \times (n-k-1)$ part is already upper Hessenberg (and its structured representation
has not changed), whereas the trailing $(j+1) \times (j+1)$ block has the 
following form: 
\begin{equation} \label{eq:AeD-tail-structure}
\begin{bmatrix}
* &*&*&\dots &*\\
x_1&t_{1,1}&t_{1,2}&\dots&t_{1,j}\\
x_2&&t_{2,2}&\dots&t_{2,j}\\
\vdots&&&\ddots&\vdots\\
x_j&&&&t_{j,j}
\end{bmatrix}.
\end{equation}
We remark that for the upper triangular matrix $T$ composing the trailing  $j \times j$ block
we have a structured representation in terms of generators. In order to continue the QR
iterations, we need to reduce this matrix again in upper Hessenberg form. 

We now assume that the deflated eigenvalues (if any) have been deflated; 
hence, we are left with a trailing $(j+1) \times (j+1)$ block $M$ which
is not in upper Hessenberg form, and needs to be reduced before 
continuing the iterations. Since this process only works on the trailing
block, we can safely ignore the rest of the matrix and its structured
parametrization which will not be altered by this reduction; only the
last $j$ or $j-1$ entries of $d,\beta,u,v$ will be updated. 

The trailing matrix has the structure of \eqref{eq:AeD-tail-structure}.
The $j \times j$ trailing upper triangular  block
is Hermitian-plus-rank-one, and we are given a structured representation
in terms of the usual vectors $d,\beta,u,v$. 
The matrix $M$ can be reduced to upper Hessenberg form directly
producing the structured representation in $\mathcal O(j)$ 
flops working as follows:
\begin{enumerate}[(i)]
	\item The entries $x_i$ of the vector $x$ composing the first 
	  column are annihilated one at a time using Givens rotations, starting from the bottom. The transformations are performed 
	  as similarities, acting on the left and right. 
\item When an entry of $x$ is annihilated (except for the last one) a bulge that breaks the Hessenberg form appears. We chase it to
the bottom using Givens rotations. This does not creates non-zero elements in $x$. 
    \item For each transformation, we update the vectors $d,\beta,u,v$
      as we have done for the QR chasing sweeps described in Section~\ref{sec:single-shift} and \ref{sec:double-shift}. 
\end{enumerate}
After $j$ steps of the above procedure, we obtain a matrix in upper
Hessenberg form, and step (iii) guarantees that we have a structured
representation of it. We note that replacing the last entries of 
$d,\beta,u,v,$ with the computed vectors yields a representation for 
the complete matrix. 

To clarify the algorithm, we pictorially describe the transformation on
a $5 \times 5$ example, which can be reduced in $6$ steps. Each arrow
represents a unitary similarity transformation by means on a Givens
rotation acting on the specified rows. 
{
\begin{align*} 
  &\begin{bmatrix}
    * & * & * & * & * \\
    \times & \times & \times & \times & \times \\
    \times &        & \times & \times & \times \\
    \times &        &        & \times & \times \\
    \times &        &        &        & \times \\
  \end{bmatrix} \xrightarrow{(4,5)} 
  \begin{bmatrix}
  * & * & * & * & * \\
  \times & \times & \times & \times & \times \\
  \times &        & \times & \times & \times \\
  \times &        &        & \times & \times \\
         &        &        & \times & \times \\
  \end{bmatrix} \xrightarrow{(3,4)} 
  \begin{bmatrix}
	* & * & * & * & * \\
	\times & \times & \times & \times & \times \\
	\times &        & \times & \times & \times \\
	       &        & \times & \times & \times \\
	       &        & \times & \times & \times \\
  \end{bmatrix}  \xrightarrow{(4,5)}  \\
  &\begin{bmatrix}
  * & * & * & * & * \\
  \times & \times & \times & \times & \times \\
  \times &        & \times & \times & \times \\
  &        & \times & \times & \times \\
  &        &        & \times & \times \\
  \end{bmatrix}  \xrightarrow{(2,3)} 
  \begin{bmatrix}
    * & * & * & * & * \\
    \times & \times & \times & \times & \times \\
           & \times & \times & \times & \times \\
           & \times & \times & \times & \times \\
           &        &        & \times & \times \\
\end{bmatrix} \xrightarrow{(3,4)} 
  \begin{bmatrix}
* & * & * & * & * \\
\times & \times & \times & \times & \times \\
       & \times & \times & \times & \times \\
       &        & \times & \times & \times \\
       &        & \times & \times & \times \\
\end{bmatrix}\xrightarrow{(4,5)} \\
&
\begin{bmatrix}
* & * & * & * & * \\
\times & \times & \times & \times & \times \\
       & \times & \times & \times & \times \\
       &        & \times & \times & \times \\
       &        &        & \times & \times \\
\end{bmatrix}.
\end{align*}}
Working with structured arithmetic requires to update only the 
diagonal, a few super- and sub-diagonal entries (as in Section~\ref{sec:single-shift}), the first column, and the vectors $u,v$. 
The upper triangular part is implicitly updated by keeping track
of these changes. 

In particular, it is never necessary to perform updates on $A_{12}$, since this 
is only implicitly determined by the four vector representation. 

\begin{remark}
  From the point of view of computational complexity
	the use of the structured  QR algorithm in AED for the computation of the Schur form of the trailing principal submatrix is not of paramount
	importance; we may compute the Schur form of $A_{22}$ by 
  an unstructured QR iteration, and then recover the updated 
  $A_{12}$ a posteriori thanks to the structure of the matrix.   
  Since the size of $A_{22}$ is expected to be negligible 
  with respect to $n$, this would not change the overall complexity. 
	Nevertheless, we prefer to work directly 
	on the structured representation using the same operations
	described in Section~\ref{sec:single-shift} because doing otherwise
	would require recovering the structure from the matrix, complicating the
	backward error analysis. Indeed, we will
	show in Section~\ref{backward} 
	how working directly on the structure can be extremely beneficial
	from the error analysis perspective. 
\end{remark}

\subsection{Parallel structured QR}\label{parallel}

In Section~\ref{AeD1} we discussed the AED procedure; this enables
the early deflation of eigenvalues, but also produces a larger number
of shifts for later use in QR sweeps 
with respect to more traditional criteria such as the Wilkinson or 
Rayleigh shifts. These can be used in several ways. 

When given $m$ shifts $\mu_1,\dots,\mu_m$ that approximate $m$  eigenvalues of $A$ we can start the QR algorithm using $\rho(x):=(x-\mu_1)\cdots(x-\mu_m)$ as shift polynomial. This process is theoretically analogous to applying consecutively  the single shift algorithm using the shifts $\mu_i$, but using a shift polynomial of degree $m$ we can deflate $m$ eigenvalues in a single step.
This can be implemented in the structured representation by mimicking
the extension done for going from single- to double- shift code, 
described in Section~\ref{sec:double-shift}. This may help with
achieving better cache usage on modern processors \cite{karlsson2014optimally}, and is indeed 
used in LAPACK's QR code by 
tightly packing several bulges
\cite{braman2002multishift,braman2002multishift2}. 

In this work we describe an alternative strategy for effectively
using these shifts, which is not trivially implemented in the
unstructured QR iteration: parallelization. We discuss how exploiting
the structure makes this task much easier. 

In practice, we choose to work with small degree shift polynomials (usually of degree 1 or 2, and hence small bulges), but we allow to start the following chasing step (with a new shift) before the end of the previous one. 
This choice allows to exploit modern processors which are often composed of multiple cores.

The scheme can be described as follows. Assume to have $p$ processors, 
with $p > 1$. The algorithm chooses a first shift\footnote{In principle, the shifts
	may be obtained using any suitable method; 
we will obtain them from the AED scheme, which provides a set of 
good shifts
that fit well in this framework.}, and starts chasing
it down to the bottom of the matrix. Before the end of the sweep, using another processor, we introduce a new bulge at the top and start chasing
that one as well. 

The idea is difficult to implement in the unstructured QR iteration, 
because of data dependencies. 
Indeed, when applying a Givens rotator on the left (resp. on the right), we modify
entries in all columns (resp. all rows) of the matrix. So there may be 
entries modified by two processors at the same time, and this creates
the need for careful data synchronization. Let us clarify this matter
with an example. Assume that a chasing step is started before the previous bulge reaches the bottom right corner. 
That is, a Givens rotator is applied on the left 
and performs a linear combination of the rows $i$ and $i+1$; at
the same time, using another processor, a Givens rotator is applied 
on the right combining the $j$-th and the $(j+1)$-th column, for $j > i + 1$. 
The entries $a_{i,j},a_{i+1,j},a_{i,j+1},a_{i+1,j+1}$ are modified at the same time by two different processors. Hence, this approach would 
require to employ expensive synchronization
techniques. 

Exploiting the vector representation automatically solves this issue. Since the upper triangular part is never explicitly updated, and we 
only need to work on the entries around the diagonal and the vectors
$u,v$, there is no need for synchronization. Hence, it is possible 
to start a number of bulges at the top separated one from the other, 
and chase them in parallel to the end of the matrix without further
complications. The differences between data dependency in 
a parallel chasing for 
the structured and unstructured QR are depicted in Figure~\ref{fig:parallel-qr}. 

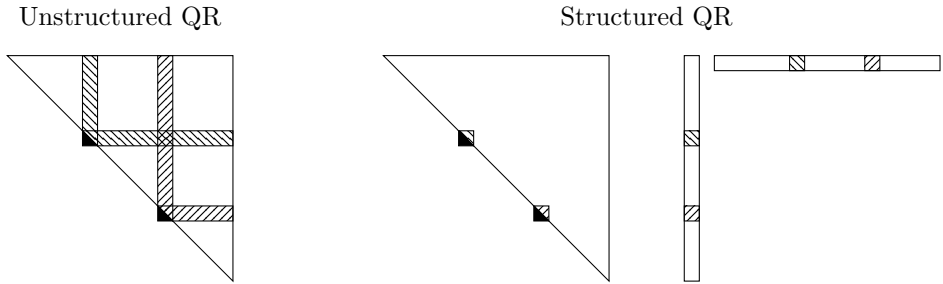
\begin{figure}
	\centering 
	\begin{tikzpicture}
	    \node at (1.5, 3.5) {Unstructured QR};
		\draw (0,3) -- (3,3) -- (3,0) -- cycle;
		\filldraw (1,2) -- (1,1.8) -- (1.2,1.8);
		\filldraw (2,1) -- (2,0.8) -- (2.2,0.8);
		\filldraw[pattern=north west lines] (1,2) -- (1.2, 1.8) --
		  (3,1.8) -- (3,2) -- cycle;
		\filldraw[pattern=north west lines] (1,2) -- (1.2,1.8) --
		  (1.2,3) -- (1.0,3) -- cycle;
		\filldraw[pattern=north east lines] 
		    (2,1) -- (2.2,0.8) -- (2.2,3) -- (2.0,3) -- cycle;		  
		\filldraw[pattern=north east lines] (2,1) -- (2.2,0.8)
		    -- (3,0.8) -- (3,1) -- cycle;
		\node at (8.5, 3.5) { Structured QR };
		\draw (5,3) -- (8,3) -- (8,0) -- cycle;
		\filldraw (6,2) -- (6,1.8) -- (6.2,1.8);
		\filldraw (7,1) -- (7,0.8) -- (7.2,0.8);
		\draw (9,0) rectangle (9.2,3);
		\draw (9.4,3) rectangle (12.4,2.8);
		\filldraw[pattern=north west lines]
		  (6,2) -- (6.2,1.8) -- (6.2,2) -- cycle;
		\filldraw[pattern=north west lines] 
		  (9,2) rectangle (9.2,1.8);
		\filldraw[pattern=north west lines] 
		  (10.4,3) rectangle (10.6,2.8);
		\filldraw[pattern=north east lines]
		  (7,1) -- (7.2,0.8) -- (7.2,1) -- cycle;
		\filldraw[pattern=north east lines] 
		  (9,1) rectangle (9.2,0.8);
		\filldraw[pattern=north east lines] 
		  (11.4,3) rectangle (11.6,2.8);		  
	\end{tikzpicture}
	\caption{Pictorial description of the parallel implementation
		of chasing multiple bulges in the unstructured and structured 
		QR. On the left, the entries in the matrix that need to be
	explicitly updated for chasing the bulges down one step in the 
unstructured QR are marked with diagonal lines; on the right, the
same is done for the structured QR, where only diagonal and subdiagonal entries (along with
the rank $1$ correction) are updated explicitly. The vectors representing
the rank $1$ correction are depicted as an outer product of tall and
thin matrices. }
\label{fig:parallel-qr}
\end{figure}

	\section{Stability analysis} \label{backward}
	
This section is devoted to the analysis of backward
stability of the proposed 
approach for computing the roots of a polynomial expressed in the
Chebyshev basis. 

To make the analysis rigorous, it is crucial to specify what 
we consider to be the input of the algorithm; for instance, the 
algorithm might be stable when
considered as an eigensolver, and unstable when considered
as a polynomial rootfinder. We comment on this fact in Section~\ref{sec:eigensolver-polynomial-stability}. 

We analyze the stability under different viewpoints; we show that the
method is stable when considered as an eigensolver, assuming that the
input is a colleague linearization of a polynomial. This yields a 
slightly stronger backward stability results compared to the original
analysis of \cite{eidelman2008efficient}, which we discuss in 
Section~\ref{sec:eigensolver-stability}. In particular, this implies 
that the method shares the same favorable properties of the unstructured QR. Then, we consider stability as a polynomial rootfinder. This is a property that the unstructured QR does not possess in general: the computed eigenvalues are eigenvalues of a closeby matrix, but not 
necessarily roots of a closeby polynomial. We show that, under an 
additional verifiable condition, we can prove that our method is stable as a polynomial rootfinder as well; this additional condition cannot be guaranteed in
general --- but is easily checkable at runtime, and it holds in 
typical cases arising from the polynomial approximations of smooth
functions on $[-1, 1]$. This is discussed in Section~\ref{sec:polynomial-stability}, and demonstrated 
on practical examples in 
Section~\ref{sec:numexp}. 

	Throughout this section, 
	we use the notation $x \lesssim \epsilon$ to mean that 
	$x$ is smaller than $\epsilon$ up to a (low-degree)
	polynomial in the size of the problem.
	The latter will be the degree of the polynomial, or the leading dimension
	of the colleague
	linearization. From now on, when not explicitly specified, 
	$\norm{\cdot}$ will be either the spectral norm (for matrices), or the
	Euclidean one (for vectors).
	
	We use the notation $\norm{p}$ to denote the norm of the vector
	containing the coefficients of the (monic) polynomial 
	under consideration. We remark that whenever we write 
	$\norm{\delta p} \lesssim \norm{p} \epsilon_m$ the constant and 
	the polynomial dependency on $n$ hidden in the $\lesssim$ notation need to be 
	independent of the polynomial $p$ under consideration. 
	If there is some additional dependency on $p$, 
	as it will happen in Section~\ref{sec:polynomial-stability}
	with the term $\hat\gamma_j(p)$, then this is 
	explicitly reported.

	\subsection{Polynomial roots through eigenvalues of colleague matrices}
	\label{sec:eigensolver-polynomial-stability}
	
	Approximating roots of polynomials by computing the eigenvalues of 
	their linearization is the method of choice in most practical
	cases. Indeed, this is the way most mathematical software implements
	commands for this task, such as for MATLAB's \texttt{roots} command. 
	
	For this reason, the question of whether computing the roots of polynomials using the 
	QR method is stable has been deeply analyzed in the literature 
	\cite{edelman1995polynomial,de2016backward,nakatsukasa2016stability,noferini2019structured,dopico2018block,noferini2017chebyshev}.
	It is known
	that the QR method implemented in floating point with the usual
	model of floating point errors applied to a matrix $A \in \mathbb C^{n \times n}$ computes a Schur form $T$ satisfying
	\[
	  Q^* (A + \delta A) Q = T, \qquad 
	  \norm{\delta A} \lesssim \norm{A} \cdot \epsilon_m. 
	\]
	If $A$ is a companion linearization for a polynomial
	expressed in the monomial basis, 
	Edelman and Murakami proved in the seminal paper \cite{eidelman2008efficient} that $A + \delta A$ has 
	as eigenvalues the roots of $p + \delta p$, with 
	$\norm{\delta p} \lesssim \norm{p}^2 \epsilon_m$. Here, both $p$ and 
	$p + \delta p$ are assumed to be monic, and we denote by $\norm{p}$ the norm of the vectors of coefficients of $p(x)$. More recently, 
	it has been shown that the same result holds for colleague linearizations as well \cite{nakatsukasa2016stability}. 
	
	This result is satisfactory if $\norm{p}$ is moderate. Otherwise, 
	one has to resort to different strategies. One possibility is to perform
	a preliminary scaling of the matrix $A$, considering a diagonal
	scaling with a matrix $D$ computed in order to minimize the norm of $D^{-1} A D$. 
	This choice is often effective, but is not guaranteed 
	to be backward stable; counterexamples where it
	gives inaccurate results can be found even when generating test
	polynomials with random coefficients 
	\cite{aurentz2018fast}. Nevertheless, this is the chosen method for 
	MATLAB's roots commands, as it behaves ``well enough'' in practice
	for polynomial with moderate norms, 
	and sometimes balancing can deliver lower forward errors than a perfectly (backward) stable algorithm. 
	
	An alternative choice is to rely on the QZ algorithm, in place of QR. This
	amounts to computing the generalized Schur form of a pencil $A - \lambda B$, 
	where $A$ and $B$ contain the coefficients of the polynomial $p(x)$; 
	this algorithm has the same properties of the QR iteration, so the computed 
	generalized Schur form $S - \lambda T$ is the exact one of $A + \delta A - \lambda (B + \delta B)$, where $\delta A, \delta B$ are relatively small compared to $A$ and $B$, respectively. This 
	guarantees that the error on the polynomial coefficients is bounded by 
	$\norm{\delta p} \lesssim \norm{p}^2 \epsilon_m$; however, 
	in this case the polynomial is not forced to be monic so it can always be scaled
	to have $\norm{p} = 1$, which solves the issue. We refer the reader to
	\cite{aurentz2018fast} for further details. 
	
	\subsection{Stability as an eigensolver}
	\label{sec:eigensolver-stability}
	
	The algorithm proposed in this work is an improved version of 
	the one in \cite{eidelman2008efficient}. For the purposes of 
	backward error analysis, we can analyze it as a sequence
	of unitary similarities described by Givens rotations; if each
	of these creates a small $\mathcal O(\epsilon_m)$ backward
	error, the whole algorithm is backward stable by composition. 	
	Indeed, we can bound the number of rotations to perform with
	a polynomial in the size of the problem by assuming that 
	the QR iteration converges in a predictable number of steps. 
	
	The improvements introduced in this work (namely, AED and 
	parallelization) rely on the same elementary operations, 
	so the backward stability result found in \cite{eidelman2008efficient}
	applies to the presented algorithm as well. 
	More specifically, we recall  \cite[Theorem~4.1]{eidelman2008efficient}
		
	\begin{theorem}[from \protect{\cite{eidelman2008efficient}}]
		\label{thm:gemignani}
		Let $A = F + uv^*$, where $F = F^*$ and $u,v \in \mathbb C^{n \times k}$. 
		Let $T$  be the Schur form computed according to the Algorithm described in Section~\ref{sec:structured-qr}. 
		Then, there exists a unitary matrix $Q$ such that 
		\[
		  T = Q^* (A + \delta A) Q, \qquad 
		  \norm{\delta A}_F \lesssim \left(
		    \norm{F}_F + \norm{u}_2 \norm{v}_2
		  \right)\epsilon_m. 
		\]
	\end{theorem}
	
	We remark that the algorithm of 
	\cite{eidelman2008efficient} is stated in more generality, and not 
	only for 
	colleague linearizations. In this more general context, it might happen\footnote{As an elementary counterexample, consider
		$F = -e_1 e_1^*$ and $u=v=e_1$, where $\norm{u}_2 = \norm{v}_2 = \norm{F}_2 = 1$, and $\norm{A}_F = 0$.} that 
	$\norm{u}_2 \norm{v}_2 \not\lesssim \norm{A}_F$, so this theorem
	does not guarantee backward stability with respect to $\norm{A}$. 
	
	Nevertheless, 
	the additional hypotheses coming from considering only 
	colleague matrices allow to state a
	stronger result. 
	
	\begin{theorem} \label{thm:stability-eigensolver}
		Let $A = F + uv^*$ be the colleague linearization of a degree $n$
		polynomial $p(x)$ expressed in the Chebyshev basis, as in 
		\eqref{eq:colleague}, and $T$ its
		approximate Schur form computed using the QR iteration
		described in Section~\ref{sec:structured-qr}. Then, 
		it exists $\delta A$ such that 
		\[
		  T = Q^* (A + \delta A) Q, \qquad 
		  \norm{\delta A} \lesssim \norm{A} \cdot \epsilon_m
		\]
	\end{theorem}

	\begin{proof}
		In view of \cite[Theorem~4.1]{eidelman2008efficient}, here
		restated as Theorem~\ref{thm:gemignani}, we have the bound 
		\[
		  \norm{\delta A} \lesssim \left(
		  \norm{F} + \norm{u} \norm{v}
		  \right)\epsilon_m,
		\]
		where we have replaced Frobenius norms with $2$-norms, since they
		are equivalent up to a polynomial in $n$. Since $F$ is the colleague
		linearization of $T_n(x)$, it has as eigenvalues the Chebyshev points
		inside $[-1, 1]$. $F$ is normal, so we have $\norm{F} \leq 1$. In addition, 
		$u = e_1$ and therefore $\norm{u} = 1$. Since for rank $1$ matrices
		and the Euclidean and spectral norm
		it holds $\norm{uv^*} = \norm{u} \norm{v}$, we have
		\[
		  uv^* = A - F \implies \norm{v} \leq 1 + \norm{A}. 
		\]		
		Combining these bounds we 
		get 
		\[
		\norm{\delta A} \lesssim \left(
		1 + \norm{v}
		\right)\epsilon_m \lesssim (2 + \norm{A}) \epsilon_m \lesssim 
		\norm{A} \epsilon_m.
		\]
	\end{proof}
	
	\subsection{Stability as a polynomial rootfinder}
	\label{sec:polynomial-stability}
	
	Theorem~\ref{thm:stability-eigensolver} guarantees that the 
	computed eigenvalues are the exact ones of a slightly
	perturbed matrix. However, if we are given a polynomial
	$p(x)$ is natural to ask if these are also the roots of 
	a closeby polynomial $p + \delta p$, satisfying $\norm{\delta p} \lesssim \norm{p} \epsilon_m$.
	
	The following result generalizes the bound on the norm of the perturbation, that holds for the unstructured QR iteration \cite{nakatsukasa2016stability}.

	\begin{theorem} \label{thm:stability-polynomial-weak}
		Let $p(x)$ be a monic polynomial in the Chebyshev basis. Then, the
		roots obtained by computing the eigenvalues of the colleague
		linearization \eqref{eq:colleague} relying on the algorithm of Section~\ref{sec:structured-qr} are the
		roots of $p(x) + \delta p(x)$ where 
		\[
		  \norm{\delta p} \lesssim \norm{p}^2 \epsilon_m, 
		\]
		where $\epsilon_m$ is the unit roundoff, and $\norm{\cdot}$ denotes
		the norm of the vector of coefficients. 
	\end{theorem}

	\begin{proof}
		Thanks to Theorem~\ref{thm:stability-eigensolver} we have that the computed eigenvalues are the ones of $A + \delta A$, with
		$\norm{\delta A} \lesssim \norm{A} \epsilon_m$. The result follows by \cite[Corollary~2.8]{nakatsukasa2016stability}. 
	\end{proof}

	If the polynomial $p(x)$ under consideration has
	coefficients of moderate norm, 
	the previous result is satisfactory. Otherwise, the quadratic
	term $\norm{p}^2$ suggests that instabilities may arise. 
	In fact, the hypothesis $\norm{p} \approx 1$ 
	is far from being satisfied in most cases of practical
	interest. Indeed, 
	if we interpolate an analytic function $f(x)$ at the Chebyshev points, as we describe in Section~\ref{sec:analytic-rootfinder}, 
	we expect the magnitude of its coefficients to decay exponentially 
	with the degree; when --- after truncation --- we normalize the polynomial to make it 
	monic the norm of its coefficient vector is bound to 
	become very large. 
	Hence, one may expect the method to be completely unreliable for
	the problem of analytic rootfinding. 
	
\begin{figure}
\centering
\subfloat[Comparison of the errors produced on the roots by the Matlab command {\tt eig} without balancing and by {\tt \fastqr}.]
  {\begin{tikzpicture}
\begin{semilogyaxis}[
  title = {},  
  ylabel = {Error},
     x tick label style={/pgf/number format/.cd,%
          scaled x ticks = false,
          set thousands separator={},
          fixed},
  legend pos=south east,
  ymajorgrids=true,
    grid style=dashed,]
    \addplot[only marks,red,mark size=.5pt,mark=square*] table [x expr=\coordindex+1, y expr=abs(\thisrowno{0})]{err_eignb.dat};
   \addlegendentry{{\tt eig\_nb}}
    \addplot[only marks,blue,mark size=.5pt,mark=*] table [x expr=\coordindex+1, y expr=abs(\thisrowno{0})]{err_fastqr.dat};
  \addlegendentry{{\tt \fastqr}}
\end{semilogyaxis}
\end{tikzpicture}}
  \qquad
\subfloat[][Comparison of the errors produced on the roots by the Matlab command {\tt eig} and by {\tt \fastqr}.]
  {\begin{tikzpicture}
\begin{semilogyaxis}[
  title = {},  
  ylabel = {Error},
  legend pos=south east,
  ymajorgrids=true,
    grid style=dashed,]
  \addplot[only marks,red,mark size=.5pt,mark=square*] table [
     x expr=\coordindex+1, y expr=abs(\thisrowno{0})]{err_eig.dat};
     \addlegendentry{{\tt eig}}
    \addplot[only marks,blue,mark size=.5pt,mark=*] table [x expr=\coordindex+1,y expr=abs(\thisrowno{0})]{err_fastqr.dat};
  \addlegendentry{{\tt \fastqr}}
\end{semilogyaxis}
\end{tikzpicture}}
  \caption{Comparison of the errors produced during the computations of the zeros of $e^x\sin(800x)$ in $[-1,1]$ . The zeros are  computed finding the roots of the Chbyshev interpolant using a $QR$ algorithm with and without balancing and the algorithm {\tt \fastqr}. The exact zeros of the function are computed analitically.}
  \label{fig:stability-test}
\end{figure}
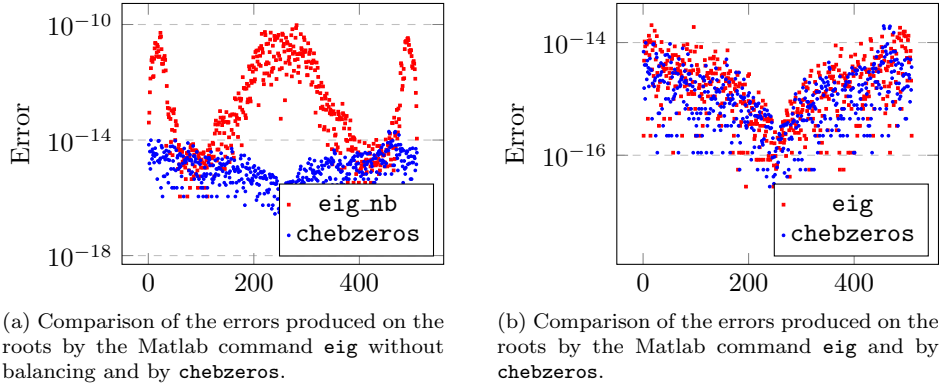

The bound given by Theorem~\ref{thm:stability-polynomial-weak}
only depends on the fact that the presented 
QR method is backward stable as an eigensolver. 
Mathematically, the method is equivalent to 
calling \texttt{eig(A, 'nobalance')} in MATLAB, 
where \texttt{A} is the colleague linearization for $p(x)$. 

We now compare these approaches on a simple
example: we approximate $f(x) = e^{x} \sin(800x)$ on $[-1, 1]$ by Chebyshev interpolation using Chebfun \cite{battles2004extension}, which yields
a polynomial of degree $891$. Then, we compute the roots using the  structured QR iteration described in Section~\ref{sec:structured-qr} (\fastqr), and the (balanced and  unbalanced) 
QR implemented
by the \texttt{eig} function in MATLAB.

The left plot of 
Figure~\ref{fig:stability-test} shows the 
absolute errors on the computed roots that are in $[-1, 1]$
(which are 
known exactly for $f(x)$). Clearly, the structured QR 
(identified by \fastqr) outperforms the unbalanced
QR (identified by \texttt{eig\_nb}) with respect
to accuracy, despite the mathematical equivalence
of these two approaches. Using balancing with \texttt{eig} 
greatly improves the accuracy, as shown by the right plot in Figure~\ref{fig:stability-test}, and yield about the same 
accuracy of the proposed algorithm (\fastqr).

The plots report forward errors, but for the theoretical
analysis we prefer to work
with backward errors instead. If $c$ is the vector
of coefficients of $p(x)$, given $n$ approximation to the roots
$y_j$, for $j = 1, \ldots, n$, 
we may define the relative backward error as follows:
\begin{equation} \label{eq:be}
  B(p; x) := \min_{\alpha \in \mathbb R} 
    \frac{\norm{c - \alpha \hat c}_2}{\norm{c}_2}, \qquad 
  p(x) = \sum_{j=0}^n c_j T_j(x), \qquad
  \prod_{j = 1}^n (x - y_j) = \sum_{j = 0}^n \hat c_j T_j(x).
\end{equation}
The number $B(p; x)$ measures the distance of the polynomial $p(x)$ to the closest
polynomial of the same degree which vanishes at the computed roots. The  backward
error for this example is around $1 \cdot 10^{-11}$
for both \fastqr\ and \texttt{eig}, but is about $5 \cdot 10^{-7}$ for 
\texttt{eig} without balancing. As already mentioned, 
this unexpected stability is in
contrast with the fact that \fastqr\ is mathematically equivalent to the
QR algorithm without balancing. 

The rest of this section is devoted to analyze the backward stability
of the approach in a polynomial sense, i.e., to derive bounds of the 
form
\begin{equation} \label{eq:polbe}
  \norm{\delta p} \lesssim C \norm{p} \epsilon_m. 
\end{equation}
We recall that, for the standard unbalanced QR iteration the results
by Edelman and Murakami \cite{edelman1995polynomial} yield 
$C \approx \norm{p}$. We show that in our approach the constant 
$C$ can be bounded with a quantity that depends on some features of 
the magnitude distributions in the vectors $u,v$ during the QR iteration, 
and this quantity is:
\begin{enumerate}[(i)]
	\item Explicitly computable: it can be given as output
	  by the algorithm at (almost) no additional cost. 
	\item Typically very small for the cases of interest, as we 
	  show in the numerical experiments.  
\end{enumerate}
Unfortunately, we can not guarantee $C$ to be always small. However, 
in our numerical experiments we manage to found only one example where
$C$ exhibits enough growth for the results to slightly deteriorate. Despite some
effort, we were not able to produce other examples that were not
small variations of the latter. 

\begin{definition}
	Let $u,v$ be two vectors in $\mathbb C^{n}$. We define  
	$\gamma_j(u, v)$ for any positive integer $j$ as the quantity:
	\[
	  \gamma_j(u,v) := \sup_{i = 1, \ldots, n-j}
	    \left\lVert \begin{bmatrix}
	    	u_i \\ \vdots \\ u_{\min\{ i+j+1, n\}}
	    	\end{bmatrix} \right\rVert_2
	    	\cdot 
	    	\left\lVert \begin{bmatrix}
	    	v_{\max\{1, i-1 \}} \\ \vdots \\ v_{i+j}
	    	\end{bmatrix} \right\rVert_2
	\]
\end{definition}
Intuitively, the above quantity bounds the norm of the submatrix
of  the 
rank one matrix $uv^*$ obtained selecting a $(j+1) \times (j+1)$ minor 
close to the diagonal. This is the submatrix whose entries are updated
during a chasing step. We denote by 
$u^{(k)}, v^{(k)}$ be the vectors obtained after
applying $k$ rotations in the algorithm described in Section~\ref{sec:structured-qr}, starting
from $u^{(0)} := u$ and $v^{(0)} := v$. We can now define
the supremum over all QR iterations of the quantities $\gamma_j(\cdot, \cdot)$, 
and we make use of the following notation:
\begin{equation} \label{eq:gammahatj}
	\hat{\gamma}_j(p)  := \sup_{k} \gamma_j(u^{(k)}, v^{(k)}), 
\end{equation}
where $k$ varies from $0$ to the number of rotations performed in 
the QR algorithm. We note that one has the trivial upper bound:
\[
  \hat{\gamma}_j(p) \leq \norm{u^{(k)}}_2 \norm{v^{(k)}}_2 
  =\norm{u}_2 \norm{v}_2  \lesssim \norm{p}, 
\]
thanks to the fact the Givens rotations preserve the norms throughout
the iterations, and $\norm{u}_2 = 1$, $\norm{v}_2 \approx \norm{p}_2$. 
We now show that the constant $C$ in \eqref{eq:polbe} can be bounded
using $\hat{\gamma}_j(p)$. In almost all cases, we will have that
$\hat\gamma_j(p) \ll \norm{p}_2$. 

To obtain this result, we build on a recent result in \cite{noferini2019structured}, which enables to bound the backward
error on $p$ by looking at the backward errors on the Hermitian
and rank $1$ part separately. We state this result (slightly simplified
and with the current notation) for 
completeness, and to better motivate the next developments. 

\begin{theorem}[from \cite{noferini2019structured}] 		
	\label{thm:chebfinal}
	Let $A = F + uv^*$ the linearization for a polynomial
	$p(x)$ expressed in the Chebyshev basis given by \eqref{eq:colleague}. 
	Consider perturbations
	$\norm{\delta F}_2 \leq \epsilon_F$, $\norm{\delta u} \leq \epsilon_u$,
	and $\norm{\delta v} \leq \epsilon_v$. Then, the matrix 
	$A + \delta A := F + \delta F + (u + \delta u) (v + \delta v)^*$ linearizes
	the polynomial $p + \delta p$, where 
	$\lVert \delta p\rVert \lesssim	
	\norm{v}_2 \epsilon_u + \epsilon_v + \norm{v}_2 \epsilon_F
	$.
\end{theorem}

In our case, bounds on the perturbation on $F,u,v$ are 
guaranteed by the following result. 

\begin{lemma} \label{lem:backward-error-structure}
	Let $C = F + uv^*$ the colleague linearization for a polynomial expressed in the Chebyshev basis. Let $\hat{\gamma}_j(p)$ be the constant
	defined in \eqref{eq:gammahatj}. If the Schur form is
	computed with the algorithm described in Section~\ref{sec:structured-qr}, 
	then it is the exact Schur form of
	\[
	  C + \delta C = F + \delta F + (u + \delta u)(v + \delta v), 
	\]
	where $\norm{\delta u}_2 \lesssim \norm{u}_2 \epsilon_m$, 
	$\norm{\delta v}_2 \lesssim \norm{v}_2 \epsilon_m$,  
	$\delta F = \delta F^*$ and 
	$\norm{\delta F} \lesssim\hat\gamma_j(p)\cdot \norm{F}_2 \epsilon_m$
	where $j = 1$
	in the single
	shift case, and $j = 2$ in the double shift one. 
\end{lemma}

\begin{proof}
	Throughout this proof, we denote by $d,\beta,u,v$ the vectors
	of the representation at a generic step of the QR iteration, 
	dropping the indices $k$. 
	
	The Hermitian-plus-rank-1 structure preservation in the backward
	error is an immediate consequence of using the generators $d,\beta,u,v$ for storing the matrix $A$. Hence, the Schur form computed according
	to the 
	algorithm described in Section~\ref{sec:structured-qr} can be written as the exact
	Schur form of a slightly perturbed matrix:
	\[
	  T = Q^* (F + \delta F + (u+\delta u)(v + \delta v)^*) Q,
	\]
	for some unitary matrix $Q$ and perturbations $\delta F 
	= \delta F^*$, $\delta u$ and $\delta v$. The bounds on the 
	backward errors on $u,v$ follows by standard backward error analysis
	since all the rotations $G$ are unitary. 
	
	The Hermitian part $F$ is not computed directly,  
	rather it is is implicitly determined by the vectors $d,\beta,u,v$. 
	Thanks
	to the symmetry, we may just prove that the backward error on
	the lower triangular is bounded. The vectors
	$d,\beta$ are contaminated with an error that can be accounted for
	in $\delta F$, and is created when applying a rotation on the
	small blocks around the diagonal as in \eqref{eq:single-shift-dbeta} (for the single shift case) and \eqref{eq:double-first-steps}, 
	\eqref{eq:double-shift-chase} for the double shift one. 
	
	In the single shift case, the updated $d,\beta$ are obtained\footnote{We have extended 
		the matrix $B$ with one more column on the left with respect to the one in \eqref{eq:single-shift-dbeta} 
	to take into account also the update to $\beta_{i-1}$, even though in Section~\ref{sec:single-shift} 
	is described as a separate step.} by 
	applying unitary operations on the $3 \times 3$ matrix $B$ that
	contains one superdiagonal and up to the third subdiagonal of the
	matrix $A$. This is computed relying on \eqref{g2.2}, and the
	backward error can be bounded by $\norm{B}_2 \epsilon_m$, up
	to a small constant. When operating on a generic block we have 
	in the single shift case
	\[
	  B = F_{i : i+2, i-1:i+1} + W, 
	\]
	where $W$ is constructed according to \eqref{eq:Aij1} and 
	\eqref{eq:Aij2} 
	using $u, v$; in particular, $W$ is obtained selecting a few
	entries from $uv^*$, and satisfies 
	$\norm{W}_2 \leq 
	\gamma_1(u, v) \leq \hat \gamma_1(p)$. 
	Using $\frac{1}{\sqrt{2}} \leq \norm{F}_2 \leq 1$ we obtain
	$
	  \norm{B} \lesssim \hat\gamma_1(p) \cdot \norm{F}_2 
	$ which in turn implies that applying Givens rotations on $B$ 
	and discarding backward errors on diagonals and subdiagonal 
	entries of $F$ yields the 
	upper bounds
	\[
	  |\delta F_{ij}| \lesssim  \hat\gamma_1(p) \cdot \norm{F}_2 
	  \epsilon_m, \qquad 
	  i \in \{ j, j +1 \}. 
	\]
	It now remains to show that the elements below the first subdiagonal, 
	which are only implicitly determined through $u,v$, respect
	a similar bound. We consider a single rotation
	$G$ acting on two consecutive
	indices $(s, s+1)$. By standard backward
	error analysis, we have that the floating point result for 
	applying  $G$ to a vector $w$ will have the form 
	$\fl{Gw} = (G + \delta G) w$ where $\norm{\delta G} \lesssim w$
	\cite{higham2002accuracy}. 
	
	Thanks to the upper Hessenberg
	structure of $A$ and $A + \delta A$,  for every $i > j + 1$, 
	we may write
	\begin{align*}
	 0 = \fl{G (A + \delta A) G^*}_{ij} &= \fl{GFG^*}_{ij} + 
	 ((G + \delta G_u) u)_i \conj{((G + \delta G_v) v)}_j.
	 \end{align*}
	 Hence, at the first order we compute the exact unitary transformation
	 of a perturbed matrix $F + \delta F + (G + \delta G_u) uv^* (G + \delta G_v)^*$, 
	 such that:
	 \begin{align*}
 0 &= F_{ij} + \delta F_{ij} +
  u_i v_j + (\delta G_u u)_i v_j + u_i (\delta G_v v)_j^* \implies \\
  |\delta F_{ij}|&\leq |(\delta G_u u)_i| |v_j| + |u_i| |(\delta G_v v)_j|,
	\end{align*}
	since $F_{ij}+u_iv_j=0.$
	 We now observe that: \[
	\norm{\delta G_u u}_2 \lesssim 
	\left\lVert \begin{bmatrix}
	u_s \\ u_{s+1}
	\end{bmatrix} \right\rVert_2 \epsilon_m, \qquad 
	\norm{\delta G_v v}_2 \lesssim 
	\left\lVert \begin{bmatrix}
	v_s \\ v_{s+1}
	\end{bmatrix} \right\rVert_2 \epsilon_m.
	\]
	Hence, we can bound the magnitude of $\delta F_{ij}$ by 
	\[
	  |\delta F_{ij}| \lesssim 
	  \left( 
	  \left\lVert \begin{bmatrix}
	  u_s \\ u_{s+1}
	  \end{bmatrix} \right\rVert_2 |v_j|  + |u_i|
	  \left\lVert \begin{bmatrix}
	  v_s \\ v_{s+1}
	  \end{bmatrix} \right\rVert_2 \right) \epsilon_m.
	\]
	Since we are in the lower triangular part below the first subdiagonal, 
	we have that 
	\[
	  F_{s:s+1,j} = -u_{s:s+1} {\conj{v}_j} \implies 
	  \left\lVert \begin{bmatrix}
	  u_s \\ u_{s+1}
	  \end{bmatrix} \right\rVert_2 |v_j| = \norm{F_{s:s+1,j}}_2 
	  \leq \norm{F}_2, 
	\]
	and similarly for the other term. Hence, for these lower 
	subdiagonal elements we can guarantee 
	$|\delta F_{ij}| \leq \norm{F}_2 \epsilon_m$, 
	and therefore we have the global bound 
	$\norm{\delta F}_2 \lesssim \hat{\gamma}_1(p) \cdot 
\norm{F}_2 \epsilon_m$.

Repeating the same steps for the double shift case yields essentially the same result, with the only exception that two upper diagonals instead of one need to be considered in \eqref{eq:double-first-steps} and \eqref{eq:double-shift-chase}; hence, we have $\norm{\delta F} \lesssim \hat{\gamma}_2(p) \cdot 
\norm{F}_2 \epsilon_m$. 
\end{proof}

We can now prove the main result, that uses the structure of the 
matrix backward error to characterize the backward error on the
polynomial coefficients. 

\begin{theorem} \label{thm:strong-stability}
	Let $C = F + uv^*$ the colleague matrix of a monic polynomial
	$p(x)$, and $\hat{\gamma}_j(p)$ defined as in \eqref{eq:gammahatj}. Then, the 
	eigenvalues computed by
	the structured QR iteration
	of Section~\ref{sec:structured-qr}
	are the 
	roots of a polynomial $p + \delta p$ satisfying 
	\[
	\norm{\delta p} \lesssim \hat{\gamma}_j(p) 
	\cdot \norm{p} \cdot \epsilon_m
	\]
	where $j = 1$ in the the single shift case, and 
	$j = 2$ for the double shift one. 
\end{theorem}

\begin{proof}
	In view of Lemma~\ref{lem:backward-error-structure}, we have the
	backward error bound $\norm{\delta F} \lesssim \hat\gamma_j(p) \epsilon_m$, where $j$ is either $1$ or $2$ depending on
	the chosen shift degree. 
	Concerning $u,v$, by standard backward error analysis of the Givens rotations 
	we obtain $\norm{\delta u} \lesssim \epsilon_m$ and 
	$\norm{\delta v} \lesssim \norm{v} \epsilon_m$, since 
	$\norm{u} = 1$. 
	We use these different bounds on the backward errors in 
	Theorem~\ref{thm:chebfinal} to obtain 
	$\norm{\delta p} \lesssim \hat{\gamma}_j(p) \norm{p} \epsilon_m$. 
\end{proof}

\section{Rootfinding for analytic functions on $[-1, 1]$}
\label{sec:analytic-rootfinder}

The tools developed in the previous section can be used for 
approximating the roots of an analytic function $f(x)$ over $[-1, 1]$. 
This can be achieved by finding the roots of the Chebyshev interpolant of the function, using the presented
	algorithm for the computation of the eigenvalues of the colleague matrix.

From a high-level perspective, the approach works as follows:
\begin{enumerate}[(i)]
	\item The function $f(x)$ is evaluated at the Chebyshev points
	  $x_i$, $i = 1, \ldots, n$, and its polynomial interpolant at
	  those points is computed using the FFT (see \cite{trefethen2019approximation}). The degree is adaptively chosen
	  to ensure that $\norm{f(x) - p(x)}_\infty \lesssim \epsilon_m$. 
	\item The roots of the interpolating polynomial are computed 
	  using the QR iteration proposed in Section~\ref{sec:double-shift}. 
	\item Among the computed roots, the ones in $[-1, 1]$ are returned
	  by the algorithm. 
\end{enumerate}
We refer to the behavior of the 
\texttt{roots} command in Chebfun
\cite{trefethen2019approximation,battles2004extension} 
for what concerns points $(i)$ and 
$(iii)$, and we employ the structured
QR iteration for addressing point
$(ii)$. 
We now discuss why we generally expect $\hat{\gamma}_j(p)$ to 
be small on such examples. Let us recall the definition 
of a Bernstein ellipse. 

\begin{definition} \label{def:bernstein-ellipse}
	Let $\rho \geq 1$. The 
	\emph{Bernstein ellipse $E_\rho$}
	is the set 
	\[
	E_\rho := \left\{ 
	\frac{z + z^{-1}}{2}, \ \ 
	|z| = \rho
	\right\} \subseteq \mathbb C. 
	\]
\end{definition}
Bernstein ellipses are the Chebyshev analogue of circles for Taylor
series. If a function is analytic inside $E_\rho$ the coefficients 
of its Chebyshev series decay as $\mathcal O(\rho^{-k})$, 
and the constant depends on the maximum of the absolute value 
of the function on $E_\rho$. We refer the reader to \cite[Chapter~8]{trefethen2019approximation} for further details. 
Remarkably, the polynomial interpolant at the Chebyshev points, 
which are numerically much easier to determine than the coefficients
of the Chebyshev series, have the same decay property, just 
weakened by a factor of $2$ (see the proof of \cite[Theorem~8.2]{trefethen2019approximation} and Chapter 4 on aliasing in the same book). 

As already pointed out, normalizing such polynomial approximant
to be monic requires to divide by a very small leading coefficient,
making the norm of polynomial coefficients very large. 
	
\section{Numerical experiments} \label{sec:numexp}

The experiments have been run on a server with two
Intel(R) Xeon(R) E5-2650v4 CPU with 12 cores and 24 threads each, 
running at 2.20 GHz,
using MATLAB R2017a with the Intel(R) Math Kernel Library Version 11.3.1. The parallel implementation is based on OpenMP. Throughout this section, we refer to 
the algorithm described in Section~\ref{sec:structured-qr} as \fastqr. 

\subsection{Accuracy}

We have tested the accuracy of 
the proposed algorithm 
by computing the roots of several
polynomials; we have tested both random polynomials of different degrees, 
and the more representative example of polynomials obtained by approximating smooth functions over $[-1, 1]$. We have also tested random
oscillatory functions generated with the Chebfun command 
\texttt{randnfun}. For \fastqr, and the 
QR iteration with and without balancing we have computed the backward
errors on the original polynomial. 

\begin{table}\scriptsize \centering
  \setlength{\tabcolsep}{3.4pt}
	\begin{tabular}{c|c|cccccc} 
    $f(x)$ & Degree & $\fastqr$ & \texttt{eig} & \texttt{eig\_nb} & $\hat\gamma_1(p)$ & $\norm{p}_2$ & \# its \\
 \hline 
$p_{100}(x)$ & $100$ & $\num{1.7e-12}$ &  $\num{7.6e-13}$ & $\num{8.6e-13}$ & $1.02$ & $5.18$ & $262$ \\ 
$p_{200}(x)$ & $200$ & $\num{1.6e-12}$ &  $\num{2.5e-12}$ & $\num{2.1e-12}$ & $\num{7.1e-1}$ & $6.88$ & $501$ \\ 
$p_{500}(x)$ & $500$ & $\num{6.1e-12}$ &  $\num{1.5e-11}$ & $\num{2.8e-11}$ & $\num{5.0e-1}$ & $\num{1.1e1}$ & $1180$ \\ 
$p_{1000}(x)$ & $1000$ & $\num{2.2e-11}$ &  $\num{1.0e-10}$ & $\num{7.9e-11}$ & $1.35$ & $\num{1.6e1}$ & $1896$ \\ 
$\log(1 + x + 10^{-3})$ & $688$ & $\num{7.7e-12}$ &  $\num{1.3e-11}$ & $\num{1.8e-2}$ & $\num{7.0e3}$ & $\num{3.5e15}$ & $1734$ \\ 
$\sqrt{x+1.01} - \sin(10^2x)$ & $180$ & $\num{7.4e-13}$ &  $\num{3.2e-13}$ & $\num{7.2e-6}$ & $\num{1.4e1}$ & $\num{1.1e15}$ & $522$ \\ 
$e^x \sin(800x)$ & $891$ & $\num{1.2e-11}$ &  $\num{9.2e-12}$ & $\num{5.7e-7}$ & $3.01$ & $\num{7.3e13}$ & $1963$ \\ 
$\sin(\frac{1}{x^2 + 10^{-2}})$ & $1430$ & $\num{1.6e-6}$ &  $\num{1.2e-11}$ & $\num{2.6e-1}$ & $\num{4.2e8}$ & $\num{3.2e15}$ & $2818$ \\ 
$\texttt{randnfun(1e-1)}$ & $711$ & $\num{3.6e-12}$ &  $\num{4.0e-12}$ & $\num{4.2e-7}$ & $2.58$ & $\num{1.1e14}$ & $1780$ \\ 
$\texttt{randnfun(1e-2)}$ & $710$ & $\num{3.8e-12}$ &  $\num{2.8e-12}$ & $\num{1.6e-7}$ & $3.47$ & $\num{8.2e13}$ & $1754$ \\ 
$\texttt{randnfun(5e-3)}$ & $1355$ & $\num{6.9e-12}$ &  $\num{6.9e-12}$ & $\num{1.6e-6}$ & $1.66$ & $\num{4.7e13}$ & $2846$ \\ 
$J_0(20x)$ & $50$ & $\num{3.3e-14}$ &  $\num{1.9e-14}$ & $\num{2.4e-14}$ & $\num{9.4e1}$ & $\num{1.2e15}$ & $154$ \\ 
$J_0(100x)$ & $148$ & $\num{1.3e-13}$ &  $\num{1.5e-13}$ & $\num{2.6e-11}$ & $\num{1.4e1}$ & $\num{4.7e14}$ & $388$ \\ 
$\frac{e^{x^2 - \frac 12} - 1}{10^{-2} + x^2}$ & $380$ & $\num{4.5e-13}$ &  $\num{1.7e-13}$ & $\num{6.5e-1}$ & $\num{1.2e3}$ & $\num{1.2e16}$ & $1192$ \\ 
$\frac{e^{x^2 - \frac 12} - 1}{10^{-4} + x^2}$ & $3632$ & $\num{3.6e-13}$ &  $\num{3.0e-13}$ & $\num{9.6e-1}$ & $\num{9.3e2}$ & $\num{1.3e16}$ & $6480$ \\
\end{tabular}
	
	\caption{Relative backward errors on the
	  polynomial coefficients obtained using different
	  methods, as defined in \eqref{eq:be}. 
	  \fastqr\ refers to the fast unbalanced QR iteration presented
	  in this paper, \texttt{eig} is the 
	  balanced QR implemented in LAPACK, as included 
  	in MATLAB, and \texttt{eig\_nb} is the unbalanced QR in LAPACK. 
  	The symbol $p_n(x)$ is used to denote a monic polynomial of degree $n$ 
  where the coefficients of degree smaller than $n$ are chosen according
  to a Gaussian distribution of $N(0, 1)$. The column
  named 
degree reports the degree of the polynomial approximant. 
\texttt{randnfun} refers to random oscillatory functions generated with
this command included in Chebfun. $J_0(x)$ denotes
the Bessel function of the first kind. The column, $\hat\gamma_1(p)$
denotes the amplification factor for the backward error
on the polynomial predicted by Theorem~\ref{thm:strong-stability}
	\label{tab:backward-errors}, $\norm{p}_2$ the norm of the
  monic polynomial, and \# its the total number of 
  iterations required for convergence.}
\end{table}

The results are summarized in Table~\ref{tab:backward-errors}. 
For each test, the method returns also the value of $\hat \rho_1(p)$
computed during the iterations. 
We see that this value is relatively
small for all functions considered, except for 
$f(x) = \sin(1 / (x^2 + 10^{-2}))$. In fact, the accuracy on this particular
example is not on par with the QR iteration. However, we remark
that the accuracy for the computed roots in $[-1, 1]$ (which are easy
to compute explicitly in this case) is at the machine precision level, 
and the only inaccurate roots are the ones close to the boundary
of the Bernstein ellipse where the function is defined. 
We have not been able to
find other examples (except ones based on modifying this particular function)
where this happens. We are still unaware if the algorithm may be
modified (for instance with a smarter shifting strategy) to avoid 
this growth. Upon closer inspection, it seems to be caused by a single
entry in the vectors $u,v$ which is growing large. 

We note that the QR without balancing performs poorly in all the cases
where analytic functions are involved, since these correspond to 
companion matrices with a large norm. The behavior is instead comparable
to the balanced case (and to \fastqr) for polynomials with random Chebyshev coefficients. 

\subsection{Performances and asymptotic complexity}

\begin{figure}
\centering
\subfloat[{Comparison of the CPU time of the algorithm implemented with and without the aggressive early deflation. Both the algorithms have been executed
	sequentially, using only one core.}\label{fig_aed_noaed}]
  {\begin{tikzpicture}
\begin{semilogxaxis}[
  title = {},  
  xlabel = {$n$},
  ylabel = {Time($s$)},
     x tick label style={/pgf/number format/.cd,%
          scaled x ticks = false,
          set thousands separator={},
          fixed},
  legend pos=north west,
  ymajorgrids=true,
    grid style=dashed,]
  \addplot[color=blue,mark=*, mark size=1pt] table {time_aed.dat};
  \addlegendentry{{\tt \fastqr}}
  \addplot[color=red, mark=square*, mark size=1pt] table {time_noaed.dat};
  \addlegendentry{{\tt fastQR}}
\end{semilogxaxis}
\end{tikzpicture}}
  \qquad
\subfloat[][Comparison of the CPU time of the algorithm {\tt \fastqr} with the CPU time taken by the Matlab command {\tt eig}. Both the algorithms have been executed using four cores.\label{fig_cheby_eig}]
  {\begin{tikzpicture}
\begin{loglogaxis}[
  title = {},  
  xlabel = {$n$},
  ylabel = {Time($s$)},
  ytick={0,10^(-6),10^(-4),10^(-2),1,10^2,10^4},
  legend pos=north west,
  ymajorgrids=true,
    grid style=dashed,]
  \addplot[color=blue,mark=*, mark size=1pt] table {time_par_definitivo.dat};
  \addlegendentry{{\tt \fastqr}}
  \addplot[color=red,mark=square*, mark size=1pt] table {time_eig_definitivo.dat};
   \addlegendentry{{\tt eig}}
\end{loglogaxis}
\end{tikzpicture}}
  \caption{}
\end{figure}

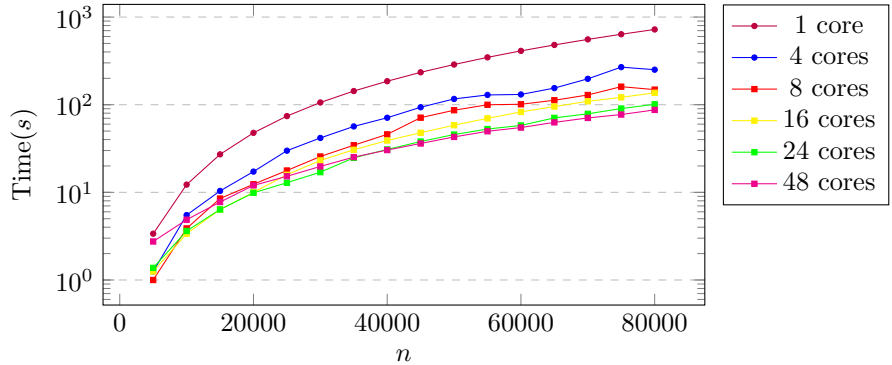
\begin{figure}
\centering
  {\begin{tikzpicture} 
\begin{semilogyaxis}[
  title = {},  
  xlabel = {$n$},
  ylabel = {Time($s$)},
     x tick label style={/pgf/number format/.cd,%
          scaled x ticks = false,
          set thousands separator={},
          fixed},
  legend pos=outer north east,
  ymajorgrids=true,
    grid style=dashed,]
    \pgfplotsset{
width=8cm, height=4cm,
scale only axis
}
    \addplot[color=purple,mark=*, mark size=1pt] table {time_np=1.dat};
  \addlegendentry{{1 core}}
  \addplot[color=blue,mark=*, mark size=1pt] table {time_np=4.dat};
  \addlegendentry{{4 cores}}
  \addplot[color=red, mark=square*, mark size=1pt] table {time_np=8.dat};
  \addlegendentry{{8 cores}}
    \addplot[color=yellow, mark=square*, mark size=1pt] table {time_np=16.dat};
  \addlegendentry{{16 cores}}
    \addplot[color=green, mark=square*, mark size=1pt] table {time_np=24.dat};
  \addlegendentry{{24 cores}}
      \addplot[color=magenta, mark=square*, mark size=1pt] table {time_np=48.dat};
  \addlegendentry{{48 cores}}
\end{semilogyaxis}
\end{tikzpicture}}
  \caption{Comparison of the clock time of  {{\tt \fastqr}} exploiting parallelism using 1, 4, 8, 16, 24 and 48 cores, for the computation of the roots of random polynomials expressed in the Chebyshev basis, with different degrees $n$.}\label{fig_parallel}
\end{figure}

We tested the speed of the proposed algorithm experimentally. The single shift iteration has been implemented in FORTRAN 90 and a MEX file has been used to interface it with MATLAB.

First, we have tested the speed up obtained 
by modifying the algorithm of \cite{eidelman2008efficient} introducing the
aggressive early deflation strategy. As 
visible in Figure~\ref{fig_aed_noaed}, 
the speed up is considerable for large matrices. More in detail, 
Figure~\ref{fig_aed_noaed} compares the CPU time of the two algorithms for randomly generated polynomials for different values of the degree $n$.

In Figure \ref{fig_cheby_eig} we compare  {\tt \fastqr} with the Matlab command {\tt eig} applied to the companion matrix in the Chebyshev basis as in \eqref{eq:colleague}. We  observe that {\tt \fastqr} is faster than {\tt eig} already for $n$ about $10$.  For polynomials of a very high degree the difference of the CPU time is remarkable. This fact is coherent with the theory since the asymptotic cost of {\tt \fastqr} is $\mathcal{O}(n^2)$ while the asymptotic cost of {\tt eig} is $\mathcal{O}(n^3).$

Then, we enabled the concurrency on the implementation as described in Section~\ref{parallel}, and again we have
tested the performance for different values 
of $n$. In order to attain optimal performances, the aggressive early deflation needs to provide a number of shifts that is 
larger than the number of available cores; hence, the dimension of the block subject to
the AED has been increased proportionally
to the number of
available cores. After some tuning, it has
been found the optimal choice to be
about $9$ times the number of available cores.

This choice produces a number of shifts that is
about $6$ times the number of parallel tasks. As the dimension of the problem decreases, thanks to deflations, the use of a large number of cores becomes less advantageous. Hence, the implementation 
halves the number of concurrent tasks (and hence the needed shifts) when the size of the matrix becomes smaller than $64$ times the number of used cores. 

In Figure \ref{fig_parallel} are reported the clock times needed for the computation using $1$, $4$, $8$, $16$, $24$ and $48$ cores. It is immediate to note the good performances of the parallel implementation already with $4$ cores, 
which brings a speed up of about $300\%$.
The gain from increasing 
the concurrency is reduced as the number of cores increases up 
to $48$. We note that the server only had $24$ combined physical cores, 
and going above $12$ required communication between the 
different CPUs, which inevitably reduces the efficiency of the 
parallelization.
Considering most consumer hardware has between $2$ and $8$ or
$16$ cores, this shows that method is well-tuned for the currently 
available architectures. 

\section{Conclusions}

A quadratic time structured QR algorithm for Hermitian-plus-rank-one matrices has been presented, by improving the ideas originally
discussed in \cite{eidelman2008efficient}.
The method relies on a structured representation for the matrix that requires to store only $4$ vectors.
We have shown that it is possible and advantageous to perform aggressive early deflation in this format, and that the structured representation allows to easily parallelize the scheme avoiding many of the difficulties present in dense  chasing 
algorithms. 

The backward stability of the method has 
been carefully analyzed, with a particular 
focus on the use of the algorithm for 
polynomial rootfinding in the Chebyshev basis. 
It has been shown that, under suitable assumptions that are easy to verify
at runtime, the method
has a small backward error on the polynomial
coefficients, a property that is not present in
the balanced and unbalanced QR iteration.
Numerical experiments confirm
the stability in the vast majority of the considered cases, and in
particular when dealing with polynomial approximating smooth functions. 
We also show that criterion is effective in detecting the cases where 
there could have been some accuracy loss. In all the cases considered 
the roots in $[-1, 1]$, which are the ones of interest, were computed
up to machine precision. 

\section*{Acknowledgments}
 Both authors are members of the INdAM/GNCS research group; the authors wish to thank the
 group for providing partial support through the
 project ``Metodi low-rank per problemi di algebra lineare con
 struttura data-sparse''. We wish to thank the referees, that 
 helped 
 to greatly improve the clarity of the manuscript. 
	
\bibliographystyle{siamplain}
\bibliography{bib_fastqr}

\end{document}